%% file: main.tex
\documentclass[12pt]{article}
\DeclareUnicodeCharacter{FFFD}{}
\input{kankyo}

\title{\textbf{Categorification of generic Su-Zhang Character Formula
}}
\author{\textbf{Shunsuke Hirota}}

\date{\textit{\today}}

\begin{document}

\maketitle
\input{ab}

\section{Introduction}\label{sec:intro}
\input{intro}

\subsection{Acknowledgments}
The author would like to thank Yoshiyuki Koga for suggesting this
project, for computing several low-rank examples, and for many helpful
discussions.
This work is supported by the Japan Society for the Promotion of Science (JSPS) through the Research Fellowship for Young Scientists (DC1), Grant Number JP25KJ1664.

\section{Well known basics}\label{sec:basics}
In this section, we summarize basic facts about $\mathfrak{gl}(m|n)$ and its
representation theory. All material in this section is
 well known.

\input{0}

\subsection{Root systems \(\Delta\)}\label{subsec:roots}

\input{1}

\subsection{Contragredient duality \(()^{\vee}\)}\label{subsec:contragredient-duality}

\input{3}

\subsection{\( \mathfrak{b} \)-Verma modules \(M^{\mathfrak{b}}(\lambda)\)}\label{subsec:verma-modules}

\input{4}

\subsection{Kac functors \(K_{\ge0}\)}\label{subsec:kac-functors}
\input{14}

\subsection{$\mathfrak{g}_{-1}$-generic weights}

\input{17}

\section{Narrow Verma modules $N^{()}(\lambda)$}

\input{28}

\section{BGG resolutions}

\input{41}

% 参考文献（小さめの文字で）
\begingroup
\small            % 好みで \footnotesize や \scriptsize でもOK
\bibliographystyle{plainnat}
\bibliography{references}
\endgroup

\noindent
\textsc{Shunsuke Hirota} \\
\textsc{Department of Mathematics, Kyoto University} \\
Kitashirakawa Oiwake-cho, Sakyo-ku, 606-8502, Kyoto \\
\textit{E-mail address}: \href{shun299509732@gmail.com}{shun299509732@gmail.com}

\end{document}

%% file: kankyo.tex
\usepackage{amsmath,amssymb,amsthm}
\usepackage{ytableau}
\usepackage{graphicx}
\usepackage{enumitem}
\usepackage{appendix}
\usepackage{tikz}
\usepackage{makecell}
\usepackage[all]{xy}
\usepackage[numbers,sort&compress]{natbib}
\usepackage[a4paper,top=1in,bottom=1in,left=1.25in,right=1.25in]{geometry}
\usetikzlibrary{arrows.meta}

\setlist[enumerate,1]{leftmargin=*}
\setlist[itemize]{leftmargin=1em}

\usepackage{hyperref}
\usepackage[capitalize,nameinlink]{cleveref}
\usepackage{aliascnt}

\theoremstyle{plain}
\newtheorem{theorem}{Theorem}[section]

\newaliascnt{lemma}{theorem}
\newtheorem{lemma}[lemma]{Lemma}
\aliascntresetthe{lemma}

\newaliascnt{proposition}{theorem}
\newtheorem{proposition}[proposition]{Proposition}
\aliascntresetthe{proposition}

\newaliascnt{corollary}{theorem}
\newtheorem{corollary}[corollary]{Corollary}
\aliascntresetthe{corollary}

\theoremstyle{definition}
\newaliascnt{definition}{theorem}
\newtheorem{definition}[definition]{Definition}
\aliascntresetthe{definition}

\newaliascnt{example}{theorem}

\aliascntresetthe{example}

\theoremstyle{remark}
\newaliascnt{remark}{theorem}

\aliascntresetthe{remark}

\newaliascnt{question}{theorem}

\aliascntresetthe{question}

\newaliascnt{conjecture}{theorem}

\aliascntresetthe{conjecture}

\crefname{theorem}{Theorem}{Theorems}       \Crefname{theorem}{Theorem}{Theorems}
\crefname{lemma}{Lemma}{Lemmas}             \Crefname{lemma}{Lemma}{Lemmas}
\crefname{proposition}{Proposition}{Propositions} \Crefname{proposition}{Proposition}{Propositions}
\crefname{corollary}{Corollary}{Corollaries} \Crefname{corollary}{Corollary}{Corollaries}
\crefname{definition}{Definition}{Definitions} \Crefname{definition}{Definition}{Definitions}
\crefname{example}{Example}{Examples}       \Crefname{example}{Example}{Examples}
\crefname{remark}{Remark}{Remarks}          \Crefname{remark}{Remark}{Remarks}
\crefname{question}{Question}{Questions}    \Crefname{question}{Question}{Questions}
\crefname{conjecture}{Conjecture}{Conjectures} \Crefname{conjecture}{Conjecture}{Conjectures}

\crefname{equation}{Equation}{Equations}    \Crefname{equation}{Equation}{Equations}
\crefname{figure}{Figure}{Figures}          \Crefname{figure}{Figure}{Figures}
\crefname{table}{Table}{Tables}             \Crefname{table}{Table}{Tables}
\renewenvironment{proof}{\noindent\textbf{Proof.}\ }{\hfill$\square$\par}

%% file: ab.tex
For semisimple Lie algebras, the BGG resolution is often viewed as a
categorification of the Weyl character formula. For general linear Lie
superalgebras, Brundan--Stroppel~\cite{brundan2010highestII} constructed an
infinite resolution of the so-called Kostant simple modules by Kac modules,
but their construction does not directly generalize the classical BGG
resolution.

In this paper we construct, for weights lying outside a neighborhood of the walls of the Weyl chambers, a
resolution that categorifies a known Weyl-type finite-sum character formula
in the same spirit as the Kac--Wakimoto formula. Our resolution is built from
images of canonical homomorphisms between Verma modules attached to
non-conjugate Borel subalgebras related by odd reflections. In particular,
the construction developed here does generalize the classical BGG resolution.

%% file: intro.tex
General linear Lie superalgebras $\mathfrak{gl}(m|n)$ are widely regarded as natural
generalizations of general linear Lie algebras, and their
representation theory has been extensively studied as a rich and
interesting subject. One of the most fundamental problems is to
determine the characters of finite-dimensional simple representations.
This problem was solved independently by Serganova
\cite{serganova1996kazhdan} and Brundan \cite{brundan2003kazhdan}.
Their results can be expressed as the following infinite sum:
\[
  \operatorname{ch} L^{()}(\lambda)
  = \sum_{\mu \in P^+} K_{\lambda,\mu}(-1)\,\operatorname{ch} K_{\ge0}(\mu),
\]
where $L^{()}(\lambda)$ denotes the simple highest weight module,
$K_{\lambda,\mu}(-1)$ is the value at $-1$ of Serganova-Brundan’s
Kazhdan--Lusztig-type polynomial (now well understood in terms of the
arc–diagram calculus of Brundan--Stroppel \cite{stroppel2012highest}),
and $K_{\ge0}(\mu)$ is the Kac module, i.e.\ the standard object in the
finite-dimensional weight module category $\mathcal{F}$.

Building on Brundan’s result \cite{brundan2003kazhdan}, Su and Zhang
\cite{su2007character} obtained, for the first time, for any finite dimensional weight modules, a Weyl-type
character formula given by a finite sum. Unlike in the purely
even Lie algebra case, however, the Su--Zhang character formula is
rather intricate and not very convenient for explicit computations.
There are also other works on character formulas for $\mathfrak{gl}(m|n)$, such as \cite{gorelik2023gruson}, \cite{chmutov2014weyl}, and \cite{sergeev2024combinatoricsirreduciblecharacterslie}. Nevertheless, in
several important classes one can obtain very neat closed character
formulas \cite{chmutov2014weyl}. The best-known example is the Kac--Wakimoto character formula
(conjectured in \cite{kac1994integrable} and proved in
\cite{chmutov2015kac,cheng2017kac,gorelik2015characters}). Very roughly,
for a so-called Kac--Wakimoto weight (a.k.a. Kostant weight or totally disconnected weight)  $\lambda$ one has
\[
  \operatorname{ch} L^{()}(\lambda)
  =
  \frac{1}{\operatorname{atyp}(\lambda)!}
  \sum_{w \in W} (-1)^{\ell(w)}
    \frac{\operatorname{ch} M^{()}(w \cdot \lambda)}
         {\displaystyle\prod_{\beta \in \Gamma_{w \cdot \lambda}} (1 + e^{-\beta})},
\]
where $M^{()}(w \cdot \lambda)$ is the Verma module, $\Gamma_{w \cdot \lambda}$ denotes the set of atypical odd positive roots for $\lambda$ and
$\operatorname{atyp}(\lambda)$ denotes the atypicality of~$\lambda$.

In the classical setting of semisimple Lie algebras, the Weyl
character formula is categorified by the BGG resolution \cite{bgg1973schubert}. It is thus
natural to ask whether one can construct a resolution that
categorifies the Kac--Wakimoto character formula. The presence of the
factor $\frac{1}{\operatorname{atyp}(\lambda)!}$, however, makes such a
categorification problem rather delicate. Note that Brundan--Stroppel
\cite{brundan2010highestII} constructed, generalizing \cite{cheng2008bgg}, for Kac--Wakimoto weights 
$\lambda$, an explicit and very interesting resolution of the Kac
module $K_{\ge0}(\lambda)$ in terms of diagram algebras. This should be
viewed as a categorification of the Serganova--Brundan infinite-sum
formula, rather than of a Weyl-type formula. On the other hand, in the setting of Lie superalgebras, it is well known among experts that finite-dimensional simple modules cannot, in general, be resolved by Verma modules \cite{cheng2008bgg}. See also \cite{coulembier2014bernstein}.

Kac--Wakimoto weights can be considered, in a sense, as the most
singular among regular dominant weights. On the other hand, there are
natural ``generic’’ classes of weights. In this paper we work with the
class of $\mathfrak{g}_{-1}$-generic weights. Our notion of
$\mathfrak{g}_{-1}$-genericity coincides with the
$\Gamma^+$–genericity of Coulembier \cite{coulembier2016bott}, and is therefore
weaker than the notion of \emph{weakly generic} and  \emph{generic} in Penkov-Serganova
\cite{coulembier2016primitive,penkov1989cohomology,penkov1994generic}.  In particular, any weight lying outside a small neighbourhood of the
walls of the Weyl chambers is $\mathfrak{g}_{-1}$–generic.

Su and Zhang \cite{su2007character} showed that, on a certain generic
region which contains all $\mathfrak{g}_{-1}$-generic weights,
their character formula simplifies to
\[
  \operatorname{ch} L^{()}(\lambda)
  =
  \sum_{w \in W} (-1)^{\ell(w)}
    \frac{\operatorname{ch} M^{()}(w \cdot \lambda)}
         {\displaystyle\prod_{\beta \in \Gamma_{w \cdot \lambda}} (1 + e^{-\beta})}.
\]
Here the troublesome factor
$\frac{1}{\operatorname{atyp}(\lambda)!}$ disappears. (See also \cite{chmutov2014weyl} for a nice exposition.) This makes it
natural to try to construct a BGG-type resolution for such~$\lambda$.

Our main point of view is that such resolutions can be constructed from odd reflections and
changes of Borel subalgebras.  
In contrast to the situation for semisimple Lie algebras, a basic
classical Lie superalgebra admits several non-conjugate Borel
subalgebras, and passing from one Borel to another via odd reflections
is a genuinely new and essential feature of the theory.  From the
perspective of Nichols algebras of diagonal type
\cite{heckenberger2009classification,andruskiewitsch2017finite} and
their classification via their Weyl groupoids (a.k.a. generalized root systems)
\cite{heckenberger2008generalization,heckenberger2020hopf,bonfert2024weyl},
these changes of Borel are also of independent interest.

Our resolutionss are defined as follows.
Let $()$ be the distinguished Borel subalgebra and
$(n^m)$ the antidistinguished Borel subalgebra of
$\mathfrak{gl}(m|n)$. 
We define the narrow Verma module $N^{()}(\lambda)$ to be the
image of up to scalar, a unique
nonzero homomorphism Verma modules sharing same character with different choice of Borels; 
\[
  N^{()}(\lambda)
  := \operatorname{Im}\!\Bigl(
    M^{()}(\lambda)
    \longrightarrow
    M^{(n^m)}\bigl(\lambda - 2\rho_{\overline{1}}\bigr)
  \Bigr).
\]

Our main result can be stated as follows.

\begin{theorem}
Let $\mathfrak{g} = \mathfrak{gl}(m|n)$ and assume that $\lambda$ is a
dominant integral and $\mathfrak{g}_{-1}$-generic weight. Then there exists an
exact sequence of $\mathfrak{g}$-modules
\[
  0 \to
  N^{()}(w_0 \cdot \lambda) \to
  \bigoplus_{\substack{w \in W \\ \ell(w) = \ell(w_0)-1}}
    N^{()}(w \cdot \lambda) \to
  \cdots \to
  \bigoplus_{\substack{w \in W \\ \ell(w) = 1}}
    N^{()}(w \cdot \lambda) \to
  N^{()}(\lambda) \to
  L^{()}(\lambda) \to 0.
\]
Moreover, for a $\mathfrak{g}_{-1}$-generic weight $\lambda$, we have,
\[
  \operatorname{ch} N^{()}(\lambda)
  =
  \frac{\displaystyle \operatorname{ch} M^{()}(\lambda)}
       {\displaystyle\prod_{\beta \in \Gamma_{w \cdot \lambda}} (1 + e^{-\beta})}.
\]
\end{theorem}

In particular, if $\lambda$ is antidominant, then $N^{()}(\lambda)$ is
simple. Hence we obtain a closed character formula for all
$\mathfrak{g}_{-1}$-generic antidominant simple highest weight
modules.

Our main tools in this paper is the restriction functor and the (dual) Kac functor.
We derive the character formula for $N^{()}(\lambda)$ from the
Su–Zhang character formula.

%% file: 0.tex
Let the base field \(k\) be an algebraically closed field of characteristic \(0\).
Let $\mathbf{Vec}$ denote the category of vector spaces, and $\mathbf{sVec}$ the category of supervector spaces.
Let $\Pi:\mathbf{sVec}\to\mathbf{sVec}$ be the \emph{parity–shift functor}.
Let $F:\mathbf{sVec}\to\mathbf{Vec}$ be the monoidal functor that forgets the $\mathbb{Z}/2\mathbb{Z}$–grading.
Note that $F$ is \emph{not} a symmetric monoidal functor.
In what follows, whenever we refer to the dimension of a supervector space,
we mean this total (ordinary) dimension:
\[
\dim V := \dim_k F(V).
\]

From now on, we will denote by \( \mathfrak{g} \) a finite dimensional Lie superalgebra.
We consider the category \( \mathfrak{g}\text{-sMod} \), where morphisms respects \(\mathbb{Z}/2\mathbb{Z}\)-grading. (This is the module category of a monoid object in \( \mathfrak{g}\text{-sMod} \) in the sense of \cite{etingof2015tensor}.)

%% file: 1.tex
\label{glmndef}
Let \( V = V_{\overline{0}} \oplus V_{\overline{1}} \) be a \( \mathbb{Z}/2\mathbb{Z} \)-graded vector space, where \( V_{\overline{0}} \) (the even part) is spanned by \( v_1, \dots, v_m \) and \( V_{\overline{1}} \) (the odd part) is spanned by \( v_{m+1}, \dots, v_{m+n} \).

The space \( \operatorname{End}(V) \) is spanned by basis elements \( E_{ij} \), defined by:  \(
E_{ij} \cdot v_k = \delta_{jk} v_i.
\)

The general linear Lie superalgebra \( \mathfrak{gl}(m|n) \) is defined as the Lie superalgebra spanned by all \( E_{ij} \) with \( 1 \leq i, j \leq m+n \), under the supercommutator:
\(
[E_{ij}, E_{kl}] = E_{ij} E_{kl} - (-1)^{|E_{ij}| |E_{kl}|} E_{kl} E_{ij},
\))
where \( |E_{ij}| = \overline{0} \) if \( E_{ij} \) acts within \( V_{\overline{0}} \) or \( V_{\overline{1}} \) (even), and \( |E_{ij}| = \overline{1} \) if it maps between \( V_{\overline{0}} \) and \( V_{\overline{1}} \) (odd).

In the general linear Lie superalgebra \( \mathfrak{g} = \mathfrak{gl}(m|n) \), the even part is given by \(
\mathfrak{g}_{\overline{0}} = \mathfrak{gl}(m) \oplus \mathfrak{gl}(n).
\)

We fix the standard Cartan subalgebra
\(
\mathfrak{h} := \bigoplus_{1 \le i \le m+n} k E_{ii}.
\)
Define linear functionals
$\varepsilon_1, \dots, \varepsilon_{n+m} \in \mathfrak{h}^*$
by requiring that
$\varepsilon_i(E_{jj}) = \delta_{ij}$ for $1 \le i,j \le n+m$.
Define \( \delta_i = \varepsilon_{m+i} \) for \( 1 \leq i \leq n \).
The non-degenerate symmetric bilinear form \( (\, , \,) \) on is defined as follows:
\(
(\varepsilon_i, \varepsilon_j) = 
(-1)^{|v_i|}\delta_{i,j}
\)

The root space \( \mathfrak{g}_\alpha \) associated with 
\( \alpha \in \mathfrak{h}^* \) is defined as
\(
\mathfrak{g}_\alpha := 
\{\, x \in \mathfrak{g} \mid [h, x] = \alpha(h)x, 
\text{ for all } h \in \mathfrak{h} \,\}.
\)
Then we have
\(
\mathfrak{g}_{\varepsilon_i - \varepsilon_j} = k E_{ij}.
\)

The set of roots \( \Delta \) is defined as
\(
\Delta := \{\, \alpha \in \mathfrak{h}^* 
\mid \mathfrak{g}_\alpha \neq 0 \,\} \setminus \{0\}.
\)

We have a root space decomposition of \( \mathfrak{g} \) with respect to \( \mathfrak{h} \):
\[
\mathfrak{g} = \mathfrak{h} 
\oplus \bigoplus_{\alpha \in \Delta} \mathfrak{g}_\alpha, 
\qquad \text{and} \quad \mathfrak{g}_0 = \mathfrak{h}.
\]

we define The sets of all even roots and odd roots as follows:
\(
\Delta_{\overline{0}} = 
\{\, \varepsilon_i - \varepsilon_j, \ \delta_i - \delta_j 
\mid i \neq j \,\},
\)
\(
\Delta_{\overline{1}}  = 
\{\, \varepsilon_i - \delta_j 
\mid 1 \le i \le m, \ 1 \le j \le n \,\}.
\)
Note that for any odd root $\alpha\in\Delta_{\bar1}$ we have $(\alpha,\alpha)=0$; i.e., every odd root is isotropic .

\begin{definition}
 A subset \( \Delta^+ \subset \Delta \) is called a \emph{positive system} if it satisfies:
\begin{enumerate}
    \item \( \Delta = \Delta^+ \cup (-\Delta^+) \) (disjoint union),
    \item For any \( \alpha, \beta \in \Delta^+ \), if \( \alpha + \beta \in \Delta \), then \( \alpha + \beta \in \Delta^+ \).
\end{enumerate}

Given a positive system \( \Delta^+ \), the associated \emph{fundamental system} is defined by
\(
\Pi := \{ \alpha \in \Delta^+ \mid \alpha \text{ cannot be written as } \alpha = \beta + \gamma \text{ with } \beta, \gamma \in \Delta^+ \}.
\)
Elements of \( \Pi \) are called \emph{simple roots}.

Let \( \Delta^+ \subset \Delta \) be a positive system. The corresponding \emph{Borel subalgebra} \( \mathfrak{b} \subset \mathfrak{g} \) is defined as
\(
\mathfrak{b} := \mathfrak{h} \oplus \bigoplus_{\alpha \in \Delta^+} \mathfrak{g}_\alpha,
\)
where \( \mathfrak{g}_\alpha \subset \mathfrak{g} \) is the root space corresponding to \( \alpha \in \Delta \).
Choose natural root vectors $e_\alpha\in\mathfrak g_\alpha$, $e_{-\alpha}\in\mathfrak g_{-\alpha}$ Then for any $\lambda\in\mathfrak h^*$,
\(
\lambda([e_\alpha,e_{-\alpha}])=(\lambda,\alpha).
\)

For a Borel subalgebra \( \mathfrak{b}\), we express the triangular decomposition of \( \mathfrak{g} \) as
\(
\mathfrak{g} = \mathfrak{n}^{\mathfrak{b}-} \oplus \mathfrak{h} \oplus \mathfrak{n}^{\mathfrak{b}+},
\)
where \( \mathfrak{b} = \mathfrak{h} \oplus \mathfrak{n}^{\mathfrak{b}+} \).

The sets of positive roots, even positive roots, and odd positive roots corresponding to \(\mathfrak b\) are denoted by
\(\Delta^{\mathfrak b,+}\),
\(\Delta_{\bar 0}^{\mathfrak b,+}\),
and
\(\Delta_{\bar 1}^{\mathfrak b,+}\),
respectively.
Let \(\Pi^{\mathfrak b}\) denote the set of all simple roots associated with the fixed Borel subalgebra \(\mathfrak b\),
and write
\(\Pi^{\mathfrak b}=\Pi^{\mathfrak b}_{\bar 0}\sqcup\Pi^{\mathfrak b}_{\bar 1}\)
for its decomposition into even and odd simple roots.
\end{definition}

The Weyl group $W:= S_m\times S_n$ of
$\mathfrak{g}_{\bar0}=\mathfrak{gl}(m)\oplus\mathfrak{gl}(n)$ acts on $\mathfrak h^*$ by
permuting the $\varepsilon_i$’s and the $\delta_j$’s separately.
For an even root $\alpha$, we denote by $s_\alpha$ the corresponding reflection in $W$.
We denote by $w_0$ the longest element of the Weyl group $W$.

We call the Borel subalgebra associated with the positive root system
$\Delta^{()+}:=\{\varepsilon_i-\varepsilon_j\mid 1\le i<j\le m\}
\cup\{\delta_p-\delta_q\mid 1\le p<q\le n\}
\cup\{\varepsilon_i-\delta_q\mid 1\le i\le m,\ 1\le q\le n\}$
the \emph{distinguished Borel subalgebra}, and denote it by $()$.
Similarly, we call the Borel subalgebra associated with
$\Delta^{(n^m)+}:=\{\varepsilon_i-\varepsilon_j\mid 1\le i<j\le m\}
\cup\{\delta_p-\delta_q\mid 1\le p<q\le n\}
\cup\{\delta_p-\varepsilon_i\mid 1\le p\le n,\ 1\le i\le m\}$
the \emph{anti-distinguished Borel subalgebra}, and denote it by $(n^m)$.

%% file: 3.tex
\begin{definition}
A weight \(\lambda\in \mathfrak{h}^*\) is \emph{integral} if \(\lambda=\sum_{i=1}^m a_i\,\varepsilon_i+\sum_{j=1}^n b_j\,\delta_j\) with  \(a_i,b_j\in\mathbb Z\)  for all \(i,j\).  
Write the set of integral weights as \(\Lambda\).

We denote by $s\mathcal{W}$ the full subcategory of $\mathfrak{g}$-modules
which are $\mathfrak{h}$-semisimple with integral weights and
finite-dimensional weight spaces (i.e.\ integral weight modules with
finite-dimensional weight spaces).

A  module \( M\in s\mathcal{W} \) admits a weight space decomposition :
\[
M = \bigoplus_{\lambda \in \Lambda} M_\lambda, \qquad
M_\lambda := \{ v \in M \mid h \cdot v = \lambda(h)v \ \text{for all } h \in \mathfrak{h} \}.
\]
\end{definition}

\begin{lemma}[See also \cite{brundan2014representations} Lemma 2.2,  \cite{chen2020primitive} Proposition 2.2.3]
We can choose \( p \in \operatorname{Map}(\Lambda, \mathbb{Z}/2\mathbb{Z}) \) such that
\[
\mathcal{W}:= \{ M \in s\mathcal{W}\mid \deg M_{\lambda} = p(\lambda) \text{ for } \lambda \in \Lambda, M_\lambda \neq 0 \}
\]
forms a Serre subcategory, \(  \mathcal{W}\) contains the trivial module and \( s\mathcal{W}= \mathcal{W}\oplus \Pi \mathcal{W}\).

\end{lemma}
Similarly, for \(\mathfrak g_{\bar 0}\) we introduce the category \(s\mathcal{W}_{\bar 0}\), and define the category \(\mathcal{W}_{\bar 0}\) so as to be compatible with the restriction functor \(\operatorname{Res}^{\mathfrak g}_{\mathfrak g_{\bar 0}}\).

\begin{lemma}[\cite{bell1993theory}, Th.\ 2.2; \cite{coulembier2017gorenstein}, §6.1]\label{indres}
Let \(\operatorname{Res}^{\mathfrak g}_{\mathfrak g_{\bar 0}}:\mathfrak g\text{-sMod}\to\mathfrak g_{\bar 0}\text{-sMod}\),
\(\operatorname{Ind}^{\mathfrak g}_{\mathfrak g_{\bar 0}}:\mathfrak g_{\bar 0}\text{-sMod}\to\mathfrak g\text{-sMod}\), and
\(\operatorname{Coind}^{\mathfrak g}_{\mathfrak g_{\bar 0}}:\mathfrak g_{\bar 0}\text{-sMod}\to\mathfrak g\text{-sMod}\)
denote restriction, induction, and coinduction, respectively; thus
\(\operatorname{Ind}^{\mathfrak g}_{\mathfrak g_{\bar 0}}\dashv \operatorname{Res}^{\mathfrak g}_{\mathfrak g_{\bar 0}}\dashv \operatorname{Coind}^{\mathfrak g}_{\mathfrak g_{\bar 0}}\).

Then:
(1) each of \(\operatorname{Res}^{\mathfrak g}_{\mathfrak g_{\bar0}}\), \(\operatorname{Ind}^{\mathfrak g}_{\mathfrak g_{\bar0}}\), \(\operatorname{Coind}^{\mathfrak g}_{\mathfrak g_{\bar0}}\) is exact and restricts between \(\mathcal W_{\bar0}\) and \(\mathcal W\);
(2) there is a natural isomorphism of functors \(\operatorname{Ind}^{\mathfrak g}_{\mathfrak g_{\bar0}}\cong \operatorname{Coind}^{\mathfrak g}_{\mathfrak g_{\bar0}}\);
(3) The unit \(\operatorname{Id}_{\mathcal O_{\bar0}}\!\to\!\operatorname{Res}^{\mathfrak g}_{\mathfrak g_{\bar0}}\!\circ\!\operatorname{Ind}^{\mathfrak g}_{\mathfrak g_{\bar0}}\) is a split monomorphism, and the counit \(\operatorname{Res}^{\mathfrak g}_{\mathfrak g_{\bar0}}\!\circ\!\operatorname{Ind}^{\mathfrak g}_{\mathfrak g_{\bar0}}\twoheadrightarrow \operatorname{Id}_{\mathcal W_{\bar0}}\) is a split epimorphism.
(4) \(\operatorname{Res}^{\mathfrak g}_{\mathfrak g_{\bar0}}\circ \operatorname{Ind}^{\mathfrak g}_{\mathfrak g_{\bar0}}\cong \Lambda(\mathfrak g_{\bar1})\otimes -\) on \(\mathcal W_{\bar0}\).
\end{lemma}

\begin{definition}
For a module \(M\) in the category \(\mathcal W\), the character \(\operatorname{ch} M\) is a formal sum that encodes the dimensions of the weight spaces of \(M\): \(\operatorname{ch} M := \sum_{\lambda} \dim M_{\lambda} e^{\lambda}\).
\end{definition}

Lie superalgebra \(\mathfrak{gl}(m|n)\) has an antiautomorphism \(\tau\) given by the formula \(\tau(E_{ij}) := -(-1)^{|i|(|i|+|j|)}\,E_{ji}\).
For a weight module \(M=\bigoplus_{\nu\in\mathfrak h^{*}} M_\nu\), we let \(M^\vee := \bigoplus_{\nu} M_\nu^{*}\) be the restricted dual of \(M\).
We may define an action of \(\mathfrak{gl}(m|n)\) on \(M^\vee\) by \((g\cdot f)(x):= -\,f(\tau(g)x)\), \(f\in M^\vee\), \(g\in \mathfrak{gl}(m|n)\), \(x\in M\).
This defines a contravariant duality functor \( (-)^\vee : \mathcal W \to \mathcal W \).

\begin{proposition}[\cite{musson2012lie}]\label{prop:duality_properties}

\begin{enumerate}
  \item $(-)^\vee$ is an exact contravariant functor.
  \item There is a natural isomorphism $M \xrightarrow{\sim} M^{\vee\vee}$ for all $M\in\mathcal W$ .
  \item $\operatorname{ch}(M^\vee)=\operatorname{ch}(M)$; in particular, the weight multiplicities are preserved.
  \item Since \(\tau\) is an automorphism of \(\mathfrak g\) compatible with the inclusions of subalgebras, there is a natural isomorphism of functors \(\mathrm{Res}_{\mathfrak g_{\bar 0}}^{\mathfrak g}\cong(\,\cdot\,)^{\vee}\circ \mathrm{Res}_{\mathfrak g_{\bar 0}}^{\mathfrak g}\circ(\,\cdot\,)^{\vee}\). Consequently, \(\mathrm{Ind}_{\mathfrak g_{\bar 0}}^{\mathfrak g}\cong(\,\cdot\,)^{\vee}\circ \mathrm{Ind}_{\mathfrak g_{\bar 0}}^{\mathfrak g}\circ(\,\cdot\,)^{\vee}\).
\end{enumerate}
\end{proposition}

%% file: 4.tex
Let  \( \mathfrak{b}\) be a Borel subalgebra.
Let \( k_{\lambda}^{\mathfrak{b}} \) be the one-dimensional \( \mathfrak{b} \)-module corresponding to \( \lambda \in \Lambda \). The \( \mathfrak{b} \)-Verma module with highest weight \( \lambda \) is defined by
\(
M^{\mathfrak{b}}(\lambda) = U(\mathfrak{g}) \otimes_{U(\mathfrak{b})} k_{\lambda}^{\mathfrak{b}}.
\)
(Here, the parity of \( k_{\lambda}^{\mathfrak{b}} \) is chosen so that \( M^{\mathfrak{b}}(\lambda) \in \mathcal{W} \). )

For the $\mathfrak{b}$-Verma supermodule $M^{\mathfrak{b}}(\lambda)$, we denote its $\mathfrak{b}$-highest weight vector by $v_{\lambda}^{\mathfrak{b}}$.
Its simple top is denoted by \( L^{\mathfrak{b}}(\lambda) \). 
Similarly, for the even part \( \mathfrak{g}_{\overline{0}} \), the corresponding Verma module and simple module are denoted by \( M_{\overline{0}}(\lambda) \) and  \( L_{\overline{0}}(\lambda) \), respectively.

We define Berezinian  weight
\(
ber:=\varepsilon_1+\cdots+\varepsilon_m-(\delta_1+\cdots+\delta_n),
\)
A weight \(\lambda\) is orthogonal to all roots if and only if \(\lambda=tber\) for some \(t\in k\) (i.e. a scalar multiple of the Berezinian weight.

\begin{definition}[integral Weyl vectors \cite{brundan2014representations}]
Define the vectors
\[
\rho_{\bar 0}:=\frac12\sum_{\beta\in\Delta_{\bar 0}^+}\beta,
\qquad
\rho_{\bar 1}^{\mathfrak b}:=\frac12\sum_{\gamma\in\Delta_{\bar 1}^{\mathfrak b,+}}\gamma,
\]
and write \(\rho_{\bar 1}:=\rho_{\bar 1}^{()}\). We have \(\rho_{\bar 1}=-\rho_{\bar 1}^{(n^m)}\).

Let the integral Weyl vector be
\[
\rho
:=
-\varepsilon_2-2\varepsilon_3-\cdots-(m-1)\varepsilon_n
+(m-1)\delta_{m+1}+(m-2)\delta_{m+2}+\cdots+(m-n)\delta_{m+n}.
\]
One checks that
\(
\rho=\rho_{\bar 0}-\rho_{\bar 1}^{()}-\frac12(m+n-1)\,\mathrm{ber}.
\)

For a Borel \(\mathfrak b\), set
\[
\rho^{\mathfrak b}
:=
\rho_{\bar 0}-\rho_{\bar 1}^{\mathfrak b}-\frac12(m+n-1)\,\mathrm{ber}.
\]

It is worth noting that for a \(\mathfrak b\)-simple root \(\alpha\),
\(
\rho^{\,r_\alpha\mathfrak b}=\rho^{\mathfrak b}+\alpha.
\)
Moreover, if \(\alpha\in\Pi^{\mathfrak b}\), then
\(
(\rho^{\mathfrak b},\alpha)=\frac12(\alpha,\alpha).
\)
In particular, if \(\alpha\) is odd isotropic, then \((\alpha,\alpha)=0\) and hence
\(
(\rho^{\mathfrak b},\alpha)=0.
\)

It is convenient to encode an integral weight \(\lambda\) by the \(m\mid n\)-tuple
\((\lambda_1,\dots,\lambda_m\mid \lambda_{m+1},\dots,\lambda_{m+n})\) of integers defined by
\(\lambda_i:=(\lambda+\rho,\varepsilon_i)\).
Note that our \(\rho^{\mathfrak b}\) is integral.

We call \(\lambda=(\lambda_1,\dots,\lambda_m\mid \lambda_{m+1},\dots,\lambda_{m+n})\) (or \(L^{()}(\lambda)\))  \emph{regular} iff
\(\lambda_i\ne\lambda_j\) for all \(1\le i<j\le m\) and  \(m+1\le i <j\le  m+n\).

For an integral weight \(\lambda=(\lambda_1,\dots,\lambda_m\mid \lambda_{m+1},\dots,\lambda_{m+n})\),
\(\lambda\) (or \(L^{()}(\lambda)\)) is \emph{antidominant} iff
\(\lambda_1\le\lambda_2\le\cdots\le\lambda_m\) and
\(\lambda_{m+1}\ge\lambda_{m+2}\ge\cdots\ge\lambda_{m+n}\).

We say that \(\lambda\)  (or \(L^{()}(\lambda)\)) is \emph{dominant} iff
\(\lambda_{1}\ge\cdots\ge\lambda_{m}\)  and
\(\lambda_{m+1}\le\cdots\le\lambda_{m+n}\) .

For $w \in W$ and a weight $\lambda \in \Lambda$, we define the
$\mathfrak{b}$-dot action of $W$ on $\Lambda$ by
\[
  w \cdot_{\mathfrak{b}} \lambda 
  := w(\lambda + \rho^{\mathfrak{b}}) - \rho^{\mathfrak{b}}.
\]
When $\mathfrak{b}=()$, we simply write $w \cdot \lambda$ for $w \cdot_{()} \lambda$.

The (usual) dot action of $W$ on $\Lambda$ is defined by
\[
  w \cdot \lambda := w(\lambda + \rho_{\overline{0}}) - \rho_{\overline{0}}.
\]

$\rho^{\mathfrak{b}}_{\overline{1}}$ is $W$–invariant iff $\mathfrak{b}=() ,(n^m)$ , so the three actions coincide:
\[
  w \cdot_{()} \lambda =  w \cdot_{(n^m)} \lambda =  w \cdot \lambda
  \qquad \text{for all } w \in W,\ \lambda \in \Lambda.
\]

For $\lambda \in \Lambda$, we denote by $\lambda^{\mathrm{dom}}$ and
$\lambda^{\mathrm{antidom}}$ the unique dominant and antidominant
weights, respectively, in the dot-orbit $W \cdot \lambda$.

 We define the \emph{atypicality} of  a simple module \( L =L^{\mathfrak b}(\lambda)\) as
\[
\operatorname{aty}L^{\mathfrak b}(\lambda) := \max \left\{ t \;\middle|\;
\begin{array}{l}
\text{there exist mutually orthogonal distinct roots } \alpha_1, \dots, \alpha_t \in \Delta_{\bar1} /(\pm 1)\\
\text{such that } (\lambda + \rho^{\mathfrak{b}}, \alpha_i) = 0 \quad \text{for all } i = 1, \dots, t
\end{array}
\right\}.
\]
This definition is independent of the choice of \( \mathfrak{b} \).
If \( \mathrm{aty}L = 0 \), then \( L \) is called \emph{typical}; otherwise, it is called \emph{atypical}.

\end{definition}

For a root $\alpha$, we write $E_\alpha$ for the natural root vector corresponding to $\alpha$.

\begin{lemma}[\cite{musson2012lie}]\label{lem:even_embedding}
Let $\mathfrak{g}=\mathfrak{gl}(m|n)$ and let $\mathfrak{b}$ be a Borel subalgebra.
Let $\alpha $ be an even $\mathfrak{b}$-simple root .
For an integral weight $\lambda \in \Lambda$ , if
\(
 \langle \lambda + \rho^{\mathfrak{b}}, \alpha^\vee \rangle
  =m \in \mathbb{Z}_{>0}\) , then
  \(
    \dim \operatorname{Hom}\bigl(M^{\mathfrak{b}}(s_\alpha \cdot_{\mathfrak{b}} \lambda),
                               M^{\mathfrak{b}}(\lambda)\bigr)=1
  \)
, and any nonzero homomorphism is injective. 
It is given by \( v_{s_\alpha \cdot_{\mathfrak{b}} \lambda}^{\mathfrak{b}} \mapsto {E_{-\alpha}^m}v_{\lambda}^{\mathfrak{b}}\)
\end{lemma}

\begin{theorem}[\cite{musson2012lie}]\label{thm:even_order_embedding}
Let $\mathfrak{b}$ be a Borel subalgebra.
Let $\lambda, \mu \in \Lambda$ be  weights such that
\(
  \lambda - \mu \;\in\; \sum_{\alpha \in \Delta_{\overline{0}}^{+}} \mathbb{Z}_{\ge 0}\,\alpha.
\)
Then
\[
  \dim \operatorname{Hom}_{\mathfrak{g}}\bigl(M^{\mathfrak{b}}(\mu),\, M^{\mathfrak{b}}(\lambda)\bigr) \;\leq\; 1,
\]
and every nonzero homomorphism
\(
  M^{\mathfrak{b}}(\mu) \longrightarrow M^{\mathfrak{b}}(\lambda)
\)
is injective.
\end{theorem}

%% file: 14.tex
Recall that we represent a Borel subalgebrawith sandard even Borel subalgebra by a partition.
In particular, we denote ()
the \emph{distinguished Borel subalgebra}, and 
$(n^m)$ the \emph{anti-distinguished Borel subalgebra}.

Define
\(
\mathfrak g_{1}:=()_{\bar1}\qquad
\mathfrak g_{-1}:=(n^m)_{\bar1}.
\)

The general linear Lie superalgebra \( \mathfrak{gl}(m|n) \) admits a natural \( \mathbb{Z} \)-grading given by
\(
\mathfrak{gl}(m|n) = \mathfrak{g}_{-1} \oplus \mathfrak{g}_{\bar0} \oplus \mathfrak{g}_{1}.
\)
This grading satisfies the relations
\(
[\mathfrak{g}_1, \mathfrak{g}_1] = [\mathfrak{g}_{-1}, \mathfrak{g}_{-1}]=0.
\)
We define   \(
\mathfrak{g}_{\geq 0} := \mathfrak{g}_{\bar0} \oplus \mathfrak{g}_1,
\quad
\mathfrak{g}_{\leq 0} := \mathfrak{g}_{-1} \oplus \mathfrak{g}_{\bar0}.
\)

Define the inflation functor \(\mathrm{Infl}_{\mathfrak g_{\bar 0}}^{\mathfrak g_{\ge 0}}\colon 
\mathfrak g_{\bar 0}\text{-sMod}\to \mathfrak g_{\ge 0}\text{-sMod}\) by letting 
\(\mathfrak g_{1}\) act trivially. 
Define the functor \((\,)^{\mathfrak g_{1}}\colon 
\mathfrak g_{\ge 0}\text{-sMod}\to \mathfrak g_{\bar 0}\text{-sMod}\) 
by \(M\mapsto M^{\mathfrak g_{1}}:=\{\,m\in M\mid \mathfrak g_{1}m=0\,\}\) . 
Then  \(\mathrm{Infl}_{\mathfrak g_{\bar 0}}^{\mathfrak g_{\ge 0}}\) is left adjoint to \((\,)^{\mathfrak g_{1}}\).

The Kac functor is defined as 
\(K_{\ge 0}:=\mathrm{Ind}_{\mathfrak g_{\ge 0}}^{\mathfrak g}\circ 
\mathrm{Infl}_{\mathfrak g_{\bar 0}}^{\mathfrak g_{\ge 0}}:
\mathfrak g_{\bar 0}\text{-sMod}\to \mathfrak g\text{-sMod}\), which is exact. 
Then \(K_{\ge 0}\) is left adjoint to the exact functor 
\(\mathrm{Res}_{\ge 0}:=(\,)^{\mathfrak g_{1}}
\circ \mathrm{Res}_{\mathfrak g_{\ge 0}}^{\mathfrak g}:\mathfrak g\text{-sMod}\to \mathfrak g_{\bar 0}\text{-sMod}\).
If we denote by \(\mathrm{Coind}_{\mathfrak g_{\ge 0}}^{\mathfrak g}\) \cite{chen2021simple}
the right adjoint of \(\mathrm{Res}_{\mathfrak g_{\ge 0}}^{\mathfrak g}\), and define  \(\mathrm{Coind}_{\ge 0}:=\mathrm{Coind}_{\mathfrak g_{\ge 0}}^{\mathfrak g}\circ 
\mathrm{Infl}_{\mathfrak g_{\bar 0}}^{\mathfrak g_{\ge 0}}:
\mathfrak g_{\bar 0}\text{-sMod}\to \mathfrak g\text{-sMod}\) then there is an isomorphism of functors
\(
\mathrm{Coind}_{\ge 0}
\circ(-\otimes L_{\bar 0}(-2\rho_{\bar 1}))\cong
K_{\ge 0}
\), see \cite{chen2021simple}.  Note that \(L_{\bar 0}(-2\rho_{\bar 1}) \cong \bigwedge^{\dim \mathfrak g_{-1}} \mathfrak g_{-1}\) as \(\mathfrak g_{\bar 0}\)-modules.

Similarly, we define the dual Kac functor 
\(K_{\le 0}:=\mathrm{Ind}_{\mathfrak g_{\le 0}}^{\mathfrak g}\circ 
\mathrm{Infl}_{\mathfrak g_{\bar 0}}^{\mathfrak g_{\le 0}}:
\mathfrak g_{\bar 0}\text{-sMod}\to \mathfrak g\text{-sMod}\),
\(\mathrm{Res}_{\le 0}:=\mathrm{Res}_{\mathfrak g_{\le 0}}^{\mathfrak g}\circ(\,\cdot\,)^{\mathfrak g_{-1}}\) and
 \(\mathrm{Coind}_{\le 0}\).

\begin{lemma}\cite{germoni1998indecomposable}
    \(
K_{\ge 0}\cong(\,\cdot\,)^{\vee}\circ
\mathrm{Coind}_{\le 0}\circ(\,\cdot\,)^{\vee},
\qquad
K_{\le 0}\cong(\,\cdot\,)^{\vee}\circ
\mathrm{Coind}_{\ge 0}\circ(\,\cdot\,)^{\vee}.
\)

\end{lemma}

\begin{proposition}\label{kacverma}
For every weight \(\lambda\) one has \(K_{\ge 0}\big(M_{\bar 0}(\lambda)\big)\cong M^{()}(\lambda)\) and \(K_{\le 0}\big(M_{\bar 0}(\lambda)\big)\cong M^{(n^{m})}(\lambda)\). 
Using \((\,\cdot\,)^{\vee}\), we have  \(\mathrm{Coind}_{\le 0}\big(M_{\bar 0}(\lambda)^{\vee}\big)\cong K_{\le 0}\big(M_{\bar 0}(\lambda-2\rho_{\bar 1})^{\vee}\big)\cong M^{()}(\lambda)^{\vee}\), and \(\mathrm{Coind}_{\ge 0}\big(M_{\bar 0}(\lambda)^{\vee}\big)\cong K_{\ge 0}\big(M_{\bar 0}(\lambda+2\rho_{\bar 1})^{\vee}\big)\cong M^{(n^{m})}(\lambda)^{\vee}\).
\end{proposition}

\begin{proposition}( \cite{chen2021simple})\label{antidom}
For a weight \(\lambda\), the following are true:
\begin{enumerate}
\item  $K_{\ge 0}(L_{\bar 0}(\lambda))^{\vee}\cong K_{\le 0}(L_{\bar 0}(\lambda-2\rho_{\bar 1}))$,
\item $\operatorname{soc}K_{\ge 0}(L_{\bar 0}(\lambda))\cong L^{(n^m)}(\lambda-2\rho_{\bar 1})$,
\item $\operatorname{top}K_{\ge 0}(L_{\bar 0}(\lambda))\cong L^{()}(\lambda)$;
\item \( \dim \operatorname{Hom}(K_{\ge 0}(L_{\bar 0}(\lambda)), K_{\le 0}(L_{\bar 0}(\mu))) = \delta_{\lambda,\mu + 2\rho_{\bar 1}} \);
\item \(K_{\ge 0}(L_{\bar 0}(\lambda))\cong  K_{\le 0}(L_{\bar 0}(\lambda- 2\rho_{\bar 1})))  \) iff  \(K_{\ge 0}(L_{\bar 0}(\lambda))\cong  L(\lambda) \) iff \(L(\lambda) \) is typical.
\end{enumerate}\end{proposition}

\begin{theorem}\cite{chen2021translated}\label{socsimple}
For every \(M\in\mathfrak g_{\bar 0}\text{-sMod}\) which is finite length, one has \(\ell(\operatorname{soc}M)=\ell(\operatorname{soc}K_{\ge 0}M)=\ell(\operatorname{soc}K_{\le 0}M)\) and \(\ell(\operatorname{top}M)=\ell(\operatorname{top}K_{\ge 0}M)=\ell(\operatorname{top}K_{\le 0}M)\), where \(\ell(-)\) denotes the length.

In particular, for every  weight $\lambda \in \Lambda$,
 $\operatorname{soc} M^{()}(\lambda)$ is simple.
\end{theorem}

\begin{proposition}[\cite{mazorchuk2014parabolic,chen2021tilting}]\label{antidom}
For a weight \(\lambda\), the following are equivalent:

     0. \(\lambda \) is antidominant.
\begin{enumerate}
\item \(M_{\bar 0}(\lambda)\cong L_{\bar 0}(\lambda)\cong M_{\bar 0}(\lambda)^{\vee}\).
\item \(P_{\bar 0}(\lambda)\cong I_{\bar 0}(\lambda)\).
\item \(M^{()}(\lambda)\cong K_{\ge 0}(L_{\bar 0}(\lambda))\).
\item \(M^{()}(\lambda)\cong M^{(n^{m})}(\lambda-2\rho_{\bar 1})^{\vee}\).

\end{enumerate}
\end{proposition}

\begin{theorem}[\cite{chen2023some}]\label{socmlam}
Let $\lambda$ be an integral weight . Then
\[
  \operatorname{Soc} M^{()}(\lambda)
  \;\cong\;
  L^{(n^m)}\bigl(\lambda^{\mathrm{antidom}}- 2\rho_{\overline{1}}\bigr).
\]
\end{theorem}

\begin{proof}
There exists an embedding
\[
  M^{()}(\lambda^{\mathrm{antidom}}) \hookrightarrow M^{()}(\lambda)
.\]

Hence
\[
  \operatorname{Soc} M^{()}(\lambda)
  \supset
  \operatorname{Soc} M^{()}(\lambda^{\mathrm{antidom}})
  \;\cong\;
  L^{(n^m)}\bigl(\lambda - 2\rho_{\overline{1}}\bigr),
\]
where the last isomorphism follows from ~\ref{antidom}.
On the other hand, by \cref{socsimple} ,
the socle $\operatorname{Soc} M^{()}(\lambda)$ is simple. 
\end{proof}

\begin{corollary}\label{cor:embed-implies-dot-orbit}
Let $\lambda,\mu \in \Lambda$ be integral weights.
If there exists an embedding
\[
  M^{()}(\mu)\;\hookrightarrow\; M^{()}(\lambda),
\]
then $\mu \in  W\cdot \lambda$.
\end{corollary}

%% file: 17.tex
To each regular dominant weight \( \lambda \in \Lambda\) (or \(L^{()}(\lambda)\)) we associate, following \cite{stroppel2012highest},  its weight diagram as following.

\(
I_{\times}(\lambda): = \{\lambda_1, \lambda_2 , \ldots, \lambda_m \}
\) and 
\(
I_{\circ}(\lambda) := \{\lambda_{m+1}, \lambda_{m+2}, \ldots,  \lambda_{m+n}\}.
\)
The integers in \( I_{\times}(\lambda) \cap I_{\circ}(\lambda) \) are labelled by \( \vee \), the remaining ones in \( I_{\times}(\lambda) \) respectively \( I_{\circ}(\lambda) \) are labelled by \( \times \) respectively \( \circ \). All other integers are labelled by a \( \wedge \).
This labelling of the number line \( \mathbb{Z} \) uniquely characterizes the weight \( \lambda \). 

\begin{definition}[Totally disconnected weights]
A regular dominant weight $\lambda \in \Lambda$ (or, equivalently, the
simple module $L^{()}(\lambda)$) is called \emph{totally disconnected} if
its weight diagram contains at least one ``$\wedge$’’ between any two
``$\vee$’’s.
\end{definition}

\begin{definition}[$\mathfrak{g}_{-1}$-generic weights]
A weight $\lambda \in \Lambda$ (or $L^{()}(\lambda)$) is called
\emph{$\mathfrak{g}_{-1}$-generic} if the set
\[
  \bigl\{\,
    \lambda - \textstyle\sum_{\alpha \in S} \alpha
    \,\bigm|\,
    S \subset \Delta_{\overline{1}}^{()+}
  \,\bigr\}
\]
lie in the same $\rho$–shifted Weyl chamber. 
\end{definition}

\begin{lemma}\label{lem:g-1-generic-properties}
Let $\lambda \in \Lambda$ be $\mathfrak{g}_{-1}$-generic. Then:
\begin{enumerate}
  \item For every $\mu \in W \cdot \lambda$ ,  $\mu$ is again $\mathfrak{g}_{-1}$-generic.
  \item $\lambda$ is regular
  \item If $\lambda$ is  dominant, then $\lambda$ is totally disconnected.
  \item For any Borel subalgebra $\mathfrak{b}\in L(m|n)$ we have
  \(
    \operatorname{Res}^{\mathfrak{g}}_{\mathfrak{g}_{\overline{0}}}
    M^{\mathfrak{b}}\bigl(\lambda + \rho^{()} - \rho^{\mathfrak{b}}\bigr)
    \;\cong\;
    \bigoplus_{S \subset \Delta_{\overline{1}}^{()+}}
      M_{\overline{0}}
      \Bigl(\lambda - \sum_{\alpha \in S} \alpha\Bigr).
  \)

\end{enumerate}
\end{lemma}
\begin{proof}
(1) and (2) are immediate from the definition.

(3) If $\lambda$ is not totally disconnected, then there exist $i$ and $j$ such that
$\lambda_i=\lambda_{m+j}=\lambda_{i+1}+1=\lambda_{m+j-1}+1$.
Then $\lambda$ and $\lambda-(\varepsilon_i-\delta_j)-(\varepsilon_{i}-\delta_{j-1})$ lie in different $\rho$--shifted Weyl chambers, contradicting $\mathfrak g_{-1}$--genericity.

(4) This follows from the fact that a Verma module has no nontrivial self-extension.
\end{proof}

%% file: 28.tex
\begin{lemma}[\cite{musson2012lie}]\label{5.2.ch_eq}
For \( \lambda ,\lambda'\in \Lambda \), the following statements hold:
\begin{enumerate}
    \item \( \operatorname{ch} M^{()}(\lambda ) = \operatorname{ch} M^{(n^m)}(\lambda - 2\rho_{\bar1}) \);
    \item \( \dim \operatorname{Hom}(M^{()}(\lambda ), M^{(n^m)}(\lambda - 2\rho_{\bar1})) = 1 \);
    \end{enumerate}
\end{lemma}

\begin{definition}[Narrow Verma modules]
We define the \emph{narrow Verma module} associated $\lambda$ by
\[
  N^{()}(\lambda)
  \;:=\;
  \operatorname{Im}\Bigl(
 M^{()}(\lambda)
  \longrightarrow
  M^{(n^m)}\bigl(\lambda - 2\rho_{\overline{1}}\bigr).
\Bigr),
\]
where we take any nonzero homomorphism
$M^{()}(\lambda)
  \longrightarrow
  M^{(n^m)}\bigl(\lambda - 2\rho_{\overline{1}}\bigr).
$

\end{definition}

The following is well known.
\begin{proposition}\label{prop:Eg1}
Set
$E_{\mathfrak{g}_1}
 := \prod_{\alpha \in \Delta_{1}^{()+}} E_\alpha \in U(\mathfrak{g}_1)$.
Then the following statements hold.
\begin{enumerate}
  \item The element $E_{\mathfrak{g}_1}$ up to scalar, does not depend on the order of
        the product $\prod_{\alpha \in \Delta_{1}^{()+}} E_\alpha$.

  \item  up to scalar, $x\,E_{\mathfrak{g}_1} = E_{\mathfrak{g}_1}\,x$ for all
        $x \in U(\mathfrak{n}_{\overline{0}}^{()-})$.

  \item For every $\lambda \in \Lambda$, the narrow Verma module
        $N^{()}(\lambda)$ is the $\mathfrak{g}$-submodule of
        $M^{(n^m)}(\lambda - 2\rho_{\overline{1}})$ generated by
        $E_{\mathfrak{g}_1} v^{(n^m)}_{\lambda - 2\rho_{\overline{1}}}$; more precisely,
        \[
          N^{()}(\lambda)
          = U(\mathfrak{g})\,E_{\mathfrak{g}_1}\,
            v^{(n^m)}_{\lambda - 2\rho_{\overline{1}}}
          \subset
          M^{(n^m)}(\lambda - 2\rho_{\overline{1}}).
        \]
\end{enumerate}
\end{proposition}

\begin{lemma}\label{lem:odd-reflection-square}
Let $\lambda \in \Lambda$ be an integral weight and let $w \in W$.
Assume we have embeddings
\[
  i_{()} \colon
  M^{()}\bigl(w\cdot\lambda\bigr)
  \longrightarrow
  M^{()}(\lambda),
\]
\[
  i_{(n^m)} \colon
  M^{(n^m)}\bigl(w\cdot\lambda - 2\rho_{\overline{1}}\bigr)
  \longrightarrow
  M^{(n^m)}\bigl(\lambda - 2\rho_{\overline{1}}\bigr),
\]
given as compositions of homomorphisms arising from even reflections.
Then the following natural diagram commutes:
\[
\begin{tikzpicture}[>=stealth]
  \node (A) at (0,-2.2)
    {$M^{()}\bigl(w\cdot\lambda \bigr)$};
  \node (B) at (0,0)
    {$M^{()}(\lambda )$};
  \node (C) at (6,0)
    {$M^{(n^m)}\bigl(\lambda - 2\rho_{\overline{1}}\bigr)$};
  \node (D) at (6,-2.2)
    {$M^{(n^m)}\bigl(w\cdot\lambda - 2\rho_{\overline{1}}\bigr)$};

  \draw[->] (A) -- node[left]  {$i_{()}$} (B);
  \draw[->] (D) -- node[right] {$i_{(n^m)}$} (C);
  \draw[->] (A) -- node[below] {$\neq 0$} (D);
  \draw[->] (B) -- node[above] {$\neq 0$} (C);
\end{tikzpicture}
\]
where the horizontal arrows denote the unique (up to scalar) nonzero
homomorphisms
\[
  M^{()}\bigl(w\cdot\lambda \bigr)
  \to
  M^{(n^m)}\bigl(w\cdot\lambda - 2\rho_{\overline{1}}\bigr),
  \qquad
  M^{()}(\lambda)
  \to
  M^{(n^m)}(\lambda - 2\rho_{\overline{1}}\bigr).
\]
\end{lemma}

\begin{proof}
It follows from \cref{prop:Eg1}
\end{proof}

\begin{proposition}\label{prop:socle-narrow}
Let $\lambda \in \Lambda$ , then,
\[
  \operatorname{Soc} N^{()}(\lambda)
  \;\cong\;
  N^{()}(\lambda^{\mathrm{antidom}})
  \;\cong\;
  L^{()}(\lambda^{\mathrm{antidom}}).
\]
\end{proposition}

\begin{proof}
It follows from \cref{socmlam}
\end{proof}

\begin{proposition}\label{prop:Res-N-narrow-summand}
Let $\lambda \in \Lambda$ be $\mathfrak{g}_{-1}$-generic.
Then,
\[
  \operatorname{Res}^{\mathfrak{g}}_{\mathfrak{g}_{\overline{0}}}
  N^{()}(\lambda)
  \quad\text{is a direct summand of}\quad
  \bigoplus_{S \subset \Delta_{\overline{1}}^{+}}
      M_{\overline{0}}
      \Bigl(\lambda - \sum_{\alpha \in S} \alpha\Bigr)
\]
as a $\mathfrak{g}_{\overline{0}}$-module.
\end{proposition}

\begin{proof}
Apply the restriction functor $\operatorname{Res}_{\overline{0}}$ to a nonzero homomorphism $M^{()}(\lambda)\to M^{(n^m)}(\lambda - 2\rho_{\overline{1}})$. Then by \cref{prop:Res-N-narrow-summand}, $ \operatorname{Res}^{\mathfrak{g}}_{\mathfrak{g}_{\overline{0}}}
  M^{()}(\lambda)$ and $ \operatorname{Res}^{\mathfrak{g}}_{\mathfrak{g}_{\overline{0}}}
  M^{(n^m)}(\lambda - 2\rho_{\overline{1}})$ decompose as direct sums of $\mathfrak g_{\overline{0}}$–Verma modules which are either isomorphic or lie in different blocks. Since the endomorphism ring of a Verma module is one–dimensional and there are no nonzero homomorphisms between Vermas in different blocks, the induced map is just a projection onto (and inclusion of) some direct summands. In particular, on the summands we care about, it is given by taking direct summands.
\end{proof}

%% file: 41.tex
\begin{definition}
Let $\mathcal A$ be an abelian category. We write $\operatorname{ch}\mathcal A$ for the category of chain complexes in $\mathcal A$,
i.e.\ objects are sequences $C^\bullet=(\cdots\to C^{i-1}\xrightarrow{d^{i-1}}C^{i}\xrightarrow{d^{i}}C^{i+1}\to\cdots)$ in $\mathcal A$
with $d^{i}\circ d^{i-1}=0$, and morphisms are chain maps.
\end{definition}

Here, we recalls classical BGG resolution.

\begin{theorem}[BGG resolution]\label{classicalBGGreso}
Let $\mathfrak g=\mathfrak{gl}(m)\oplus\mathfrak{gl}(n)$ and let $\lambda$ be regular dominant integral.
Then there exists an exact sequence $BGG_{\overline{0}}(\lambda)$ of $\mathfrak g$--modules of the form
\[
0\to M_{\overline{0}}(w_0\cdot\lambda)\to
\bigoplus_{\ell(w)=\ell(w_0)-1} M_{\overline{0}}(w\cdot\lambda)\to\cdots\to
\bigoplus_{\ell(w)=1} M_{\overline{0}}(w\cdot\lambda)\to
M_{\overline{0}}(\lambda)\to L_{\overline{0}}(\lambda)\to 0.
\]
Moreover, $BGG_{\overline{0}}(\lambda)$ is indecomposable.
\end{theorem}

The following is our main result.

\begin{theorem}\label{narrowBGG}
Let $\mathfrak{g} = \mathfrak{gl}(m|n)$ and assume that $\lambda$ is
dominant integral and $\mathfrak{g}_{-1}$-generic. Then there exists an exact
sequence of $\mathfrak{g}$-modules
\[
  0 \to
  N^{()}(w_0 \cdot \lambda) \to
  \bigoplus_{\substack{w \in W \\ \ell(w) = \ell(w_0)-1}} N^{()}(w \cdot \lambda) \to
  \cdots \to
  \bigoplus_{\substack{w \in W \\ \ell(w) = 1}} N^{()}(w \cdot \lambda) \to
  N^{()}(\lambda) \to
  L^{()}(\lambda) \to 0.
\]
\end{theorem}
\begin{proof}
Applying the Kac functor to \cref{classicalBGGreso}, we obtain the following exact
sequence of $\mathfrak{g}$-modules \(K_{\ge 0}\mathrm{BGG}_{\bar0}(\lambda)\):
\[
\begin{aligned}
  0 &\to
  M^{()}(w_0 \cdot \lambda)
  \to
  \bigoplus_{\substack{w \in W \\ \ell(w) = \ell(w_0)-1}} M^{()}(w \cdot \lambda)
  \\
    &\to
  \cdots \to
  \bigoplus_{\substack{w \in W \\ \ell(w) = 1}} M^{()}(w \cdot \lambda)
  \to
  M^{()}(\lambda)\to
  K_{\ge 0}\bigl(L_{\overline{0}}(\lambda)\bigr)
  \to 0.
\end{aligned}
\]

Applying the dual Kac functor to the classical BGG resolution
\cref{classicalBGGreso}, we obtain the following exact
sequence of $\mathfrak{g}$-modules \(K_{\le 0}\mathrm{BGG}_{\overline{0}}(\lambda- 2\rho_{\overline{1}})\):

\[
\begin{aligned}
  0 &\to
  M^{(n^m)}\bigl(w_0 \cdot (\lambda - 2\rho_{\overline{1}})\bigr)
  \to
  \bigoplus_{\substack{w \in W \\ \ell(w) = \ell(w_0)-1}}
    M^{(n^m)}\bigl(w \cdot (\lambda - 2\rho_{\overline{1}})\bigr)
  \\
    &\to
  \cdots \to
  \bigoplus_{\substack{w \in W \\ \ell(w) = 1}}
    M^{(n^m)}\bigl(w \cdot (\lambda - 2\rho_{\overline{1}})\bigr)
  \to
  M^{(n^m)}\bigl(\lambda - 2\rho_{\overline{1}}\bigr)
  \to
  K_{\le 0}\bigl(L_{\overline{0}}(\lambda - 2\rho_{\overline{1}})\bigr)
  \to 0.
\end{aligned}
\]

By \cref{lem:odd-reflection-square}, we obtain the hom of complex \(K_{\ge 0}\mathrm{BGG}_{\bar0}(\lambda)\to  K_{\le 0}\mathrm{BGG}_{\overline{0}}(\lambda- 2\rho_{\overline{1}})\).

We obtain the image denoted by \(\mathrm{BGG}^{()}(\lambda)\):
\[
\begin{aligned}
  0 &\to
  N^{()}(w_0 \cdot \lambda)
  \to
  \bigoplus_{\substack{w \in W \\ \ell(w) = \ell(w_0)-1}}
    N^{()}(w \cdot \lambda)
  \\
    &\to
  \cdots \to
  \bigoplus_{\substack{w \in W \\ \ell(w) = 1}}
    N^{()}(w \cdot \lambda)
 \xrightarrow{f}
  N^{()}(\lambda)
\to
 \operatorname{Cok}(f)
  \to 0.
\end{aligned}
\]

In particular,  we obtain following commutative diagram.

\begin{tikzpicture}[xscale=2,yscale=1,baseline=(A11.base)]
  % 上段
  \node (A11) at (0,2) {$M^{()}(\lambda)$};
  \node (A12) at (3,2) {$N^{()}(\lambda)$};
  \node (A13) at (6,2) {$M^{(n^m)}\bigl(\lambda - 2\rho_{\overline{1}}\bigr)$};

  % 中段
  \node (A21) at (0,1) {$K_{\ge 0}\bigl(L_{\overline{0}}(\lambda)\bigr)$};
  \node (A22) at (3,1) {$\operatorname{Cok}(f)$};
  \node (A23) at (6,1) {$K_{\le 0}\bigl(L_{\overline{0}}(\lambda - 2\rho_{\overline{1}})\bigr)$};

  % 下段
  \node (A31) at (0,0) {$0$};
  \node (A32) at (3,0) {$0$};
  \node (A33) at (6,0) {$0$};

  % 横方向の矢印
  \draw[->] (A11) -- (A12);
  \draw[->] (A12) -- (A13);
  \draw[->] (A21) -- (A22);
  \draw[->] (A22) -- (A23);

  % 縦方向の矢印
  \draw[->] (A11) -- (A21);
  \draw[->] (A21) -- (A31);
  \draw[->] (A12) -- (A22);
  \draw[->] (A22) -- (A32);
  \draw[->] (A13) -- (A23);
  \draw[->] (A23) -- (A33);
\end{tikzpicture}

Applying the restriction functor $\mathrm{Res}_{\mathfrak g_{\bar 0}}^{\mathfrak g}$, we obtain a morphism of
complexes
\[
  \mathrm{Res}_{\mathfrak g_{\bar 0}}^{\mathfrak g}\bigl(K_{\ge 0}\mathrm{BGG}_{\overline{0}}(\lambda)\bigr)
  \;\longrightarrow\;
  \mathrm{Res}_{\mathfrak g_{\bar 0}}^{\mathfrak g}\bigl(\mathrm{BGG}^{()}(\lambda)\bigr)
  \;\longrightarrow\;
  \mathrm{Res}_{\mathfrak g_{\bar 0}}^{\mathfrak g}\bigl(K_{\le 0}\mathrm{BGG}_{\overline{0}}(\lambda - 2\rho_{\overline{1}})\bigr).
\]

By \cref{lem:g-1-generic-properties}, \[
  \operatorname{Res}^{\mathfrak g}_{\mathfrak g_{\bar 0}}
    \bigl(K_{\ge 0}\mathrm{BGG}_{\overline{0}}(\lambda)\bigr)
  \;\cong\;
  \operatorname{Res}^{\mathfrak g}_{\mathfrak g_{\bar 0}}
    \bigl(K_{\le 0}\mathrm{BGG}_{\overline{0}}(\lambda - 2\rho_{\overline{1}})\bigr)
  \;\cong\;
   \bigoplus_{S \subset \Delta_{\overline{1}}^{()+}}
     \mathrm{BGG}_{\overline{0}} \Bigl(\lambda - \sum_{\alpha \in S} \alpha\Bigr).
\]

By \cref{lem:odd-reflection-square} and \cref{prop:Res-N-narrow-summand},
the complex $\mathrm{Res}_{\mathfrak g_{\bar 0}}^{\mathfrak g}\bigl(\mathrm{BGG}^{()}(\lambda)\bigr)$ is a
direct summand of both
$\mathrm{Res}_{\mathfrak g_{\bar 0}}^{\mathfrak g}\bigl(K_{\ge 0}\mathrm{BGG}_{\overline{0}}(\lambda)\bigr)$
and
$\mathrm{Res}_{\mathfrak g_{\bar 0}}^{\mathfrak g}\bigl(K_{\le 0}\mathrm{BGG}_{\overline{0}}(\lambda - 2\rho_{\overline{1}})\bigr)$.

Hence the canonical homomorphism
\[
   \operatorname{Cok}(f)
  \longrightarrow
  K_{\le 0}\bigl(L_{\overline{0}}(\lambda - 2\rho_{\overline{1}})\bigr)
\]
is injective.

In particular,
\[
   \operatorname{Cok}(f)
  \;\cong\;
  L^{()}(\lambda).
\]

The exactness of
$\mathrm{Res}_{\mathfrak g_{\bar 0}}^{\mathfrak g}\bigl(\mathrm{BGG}^{()}(\lambda)\bigr)$ follows from the
exactness of $\bigoplus_{S \subset \Delta_{\overline{1}}^{()+}}
     \mathrm{BGG}_{\overline{0}} \Bigl(\lambda - \sum_{\alpha \in S} \alpha\Bigr)$.

Finally, the exactness of $\mathrm{BGG}^{()}(\lambda)$ itself follows from the
exactness of
$\mathrm{Res}_{\mathfrak g_{\bar 0}}^{\mathfrak g}\bigl(\mathrm{BGG}^{()}(\lambda)\bigr)$.

\end{proof}

Denote by $\Gamma_\lambda$ the set of atypical ()-positive roots of $\lambda$:
\[
  \Gamma_\lambda
  = \bigl\{\, \alpha\in\Delta_{\bar{1}}^{()+},\bigm|\,
      (\lambda + \rho,\ \alpha) = 0
    \bigr\}.
\]

Note that \(
  w \Gamma_\lambda = \Gamma_{w \cdot \lambda}
\) for all \(w \in W\).

\begin{theorem}[\cite{su2007character,chmutov2014weyl}]\label{thm:char-totally-disconnected}
Let $\lambda \in \Lambda$ be totally disconnected.
Then,
\begin{align*}
  \operatorname{ch} L^{()}(\lambda)
  &=
  \frac{\displaystyle\prod_{\beta \in \Delta_{\overline{1}}^{()+}} (1 + e^{-\beta})}
       {\displaystyle\prod_{\gamma \in \Delta_{\overline{0}}^{+}} (1 - e^{-\gamma})}
  \sum_{w \in W} (-1)^{\ell(w)}
    \frac{ e^{w \cdot \lambda} }
         {\displaystyle\prod_{\beta \in \Gamma_{w \cdot \lambda}} (1 + e^{-\beta})}
  \\
  &=
  \sum_{w \in W} (-1)^{\ell(w)} e^{w \cdot \lambda}\,
  \frac{\displaystyle\prod_{\beta \in \Delta_{\overline{1}}^{()+} \setminus \Gamma_{w \cdot \lambda}} (1 + e^{-\beta})}
       {\displaystyle\prod_{\gamma \in \Delta_{\overline{0}}^{+}} (1 - e^{-\gamma})}
  \\
  &=
  \sum_{w \in W} (-1)^{\ell(w)}
    \frac{\operatorname{ch} M^{()}(w \cdot \lambda)}
         {\displaystyle\prod_{\beta \in \Gamma_{w \cdot \lambda}} (1 + e^{-\beta})}.
\end{align*}
\end{theorem}

\begin{theorem}\label{thm:char-narrow}
Let $\lambda \in \Lambda$ be $\mathfrak{g}_{-1}$-generic.
Then the character of the narrow Verma module $N^{()}(\lambda)$ is
\[
  \operatorname{ch} N^{()}(\lambda)
  = e^{\lambda}\,
     \frac{\displaystyle\prod_{\beta \in \Delta_{\overline{1}}^{()+} \setminus \Gamma_{w \cdot \lambda}}
                   (1 + e^{-\beta})}
          {\displaystyle\prod_{\gamma \in \Delta_{\overline{0}}^{+}}
                   (1 - e^{-\gamma})} \\
  = \frac{\displaystyle \operatorname{ch} M^{()}(\lambda)}
          {\displaystyle\prod_{\beta \in \Gamma_{w \cdot \lambda}} (1 + e^{-\beta})}.
\]
\end{theorem}
\begin{proof}
Let $\lambda \in \Lambda$ be dominant.
By \cref{lem:g-1-generic-properties}, the weight $\lambda$ is totally
disconnected. Hence, by \cref{thm:char-totally-disconnected} we have
\[
\begin{aligned}
  \operatorname{ch} L^{()}(\lambda)
  &= \operatorname{ch}\,\mathrm{Res}_{\mathfrak g_{\bar 0}}^{\mathfrak g} L^{()}(\lambda) \\
  &= \sum_{w \in W} (-1)^{\ell(w)}
    \frac{\operatorname{ch} M^{()}(w \cdot \lambda)}
         {\displaystyle\prod_{\beta \in \Gamma_{w \cdot \lambda}} (1 + e^{-\beta})} \\
 &=\sum_{w \in W} (-1)^{\ell(w)}
    \sum_{S \subset \Delta_{\overline{1}}^{()+} \setminus \Gamma_{w \cdot \lambda}}
      \operatorname{ch}
      M_{\overline{0}}
      \Bigl(w \cdot \lambda - \sum_{\alpha \in S} \alpha\Bigr).
      \end{aligned}
\]

On the other hand, by \cref{narrowBGG} we obtain
\[
  \operatorname{ch} L^{()}(\lambda)
  = \sum_{w \in W} (-1)^{\ell(w)}\,
    \operatorname{ch} N^{()}(w \cdot \lambda).
\]

By \cref{prop:Res-N-narrow-summand}, each
$\mathrm{Res}_{\mathfrak g_{\bar 0}}^{\mathfrak g} N^{()}(w \cdot \lambda)$ is a direct summand of
\[
  \bigoplus_{S \subset \Delta_{\overline{1}}^{()+} \setminus \Gamma_{w \cdot\lambda}}
      M_{\overline{0}}
      \Bigl(w \cdot \lambda - \sum_{\alpha \in S} \alpha\Bigr).
\]

Since the subspaces
$V_w:=\mathrm{span}\Bigl\{\operatorname{ch} M_{\overline{0}}\bigl(w\cdot\lambda-\sum_{\alpha\in S}\alpha\bigr)\ \Bigm|\ 
S\subset \Delta_{\overline{1}}^{()+}\setminus \Gamma_{w\cdot\lambda}\Bigr\}$
(for varying $w$) are linearly independent, this yields the desired identity and completes the proof.
\end{proof}

\begin{corollary}\label{cor:char-g-1-generic-antidom}
Let $\lambda \in \Lambda$ be $\mathfrak{g}_{-1}$-generic and
antidominant.
Then
\[
  \operatorname{ch} L^{()}(\lambda)
  = e^{\lambda}\,
     \frac{\displaystyle\prod_{\beta \in \Delta_{\overline{1}}^{()+} \setminus \Gamma_{w \cdot \lambda}}
                   (1 + e^{-\beta})}
          {\displaystyle\prod_{\gamma \in \Delta_{\overline{0}}^{+}}
                   (1 - e^{-\gamma})} \\
  = \frac{\displaystyle \operatorname{ch} M^{()}(\lambda)}
          {\displaystyle\prod_{\beta \in \Gamma_{w \cdot \lambda}} (1 + e^{-\beta})}.
\]
\end{corollary}